\documentclass{article}%
\usepackage{amsmath}
\usepackage{amsfonts}
\usepackage{amssymb}
\usepackage{graphicx}%
\providecommand{\U}[1]{\protect\rule{.1in}{.1in}}
\newtheorem{theorem}{Theorem}[section]

\newtheorem{definition}[theorem]{Definition}
\newtheorem{example}[theorem]{Example}

\newtheorem{lemma}[theorem]{Lemma}

\newtheorem{proposition}[theorem]{Proposition}

\newenvironment{proof}[1][Proof]{\noindent\textbf{#1.} }{\ \rule{0.5em}{0.5em}}
\begin{document}

\title{Ground state solutions for a semilinear elliptic problem\\with critical-subcritical growth}
\author{C.O. Alves$^{\text{a}}$, G. Ercole$^{\text{b}}$\thinspace, M.D. Huam\'{a}n
Bola\~{n}os$^{\text{b}}$\\{\small {$^{\mathrm{a}}$} Universidade Federal de Campina Grande, Campina
Grande, PB, 58.109-970, Brazil. }\\{\small E-mail: coalves@mat.ufcg.edu.br}\\{\small {$^{\mathrm{b}}$} Universidade Federal de Minas Gerais, Belo
Horizonte, MG, 30.123-970, Brazil.}\\{\small E-mail:grey@mat.ufmg.br }\\{\small E-mail: emdi\_82@hotmail.com}}
\maketitle

\begin{abstract}
In this work, we study the of positive ground state solution for the semilinear elliptic problem
\[
\left\{
\begin{array}
[c]{ll}%
-\Delta u=u^{p(x)-1},\quad u>0 & \mathrm{in}\,G\subseteq\mathbb{R}^{N}%
,\,N\geq3\\
u\in D_{0}^{1,2}(G), &
\end{array}
\right.
\]
where $G$ is either $\mathbb{R}^{N}$ or a bounded domain, and $p:G\rightarrow
\mathbb{R}$ is a continuous function assuming critical and subcritical values.

\end{abstract}

\section{Introduction}

This paper concerns with the existence of positive ground state solutions for the
semilinear elliptic problem%

\begin{equation}
\left\{
\begin{array}
[c]{ll}%
-\Delta u=u^{p(x)-1},\,u>0 & \mathrm{in}\quad G,\\
u\in D_{0}^{1,2}(G), &
\end{array}
\right.  \tag{P}\label{int1}%
\end{equation}
where either $G=\mathbb{R}^{N}$ and $D_{0}^{1,2}(G)=D^{1,2}(\mathbb{R}^{N})$
or $G$ is a bounded domain in $\mathbb{R}^{N}$ and $D_{0}^{1,2}(G)=H_{0}%
^{1}(G).$ In both cases, $N\geq3$ and $p:G\rightarrow\mathbb{R}$ is a
continuous function satisfying the following condition:

\begin{description}
\item[($\mathrm{H}_{1}$)] There exist a bounded set $\Omega\subset G,$ with
positive $N$-dimensional Lebesgue measure, and positive constants $p^{-},$
$p^{+},$ and $\delta$ such that:%
\begin{equation}
2<p^{-}\leq p(x)\leq p^{+}<2^{\ast}\quad\forall\,x\in\Omega, \tag{$H_1a$}%
\end{equation}%
\begin{equation}
p(x)\equiv2^{\ast}\quad\forall\,x\in G\setminus\Omega_{\delta},\quad
\tag{$H_1b$}%
\end{equation}%
and 
\begin{equation}
2<p^{-}\leq p(x)<2^{\ast}\quad\forall\,x\in\Omega_{\delta}, \tag{$H_1c$}%
\end{equation}
where
\[
\Omega_{\delta}:=\{x\in G:\operatorname{dist}(x,\overline{\Omega})\leq
\delta\}.
\]
$\quad$
\end{description}

There are several works in the literature dealing with semilinear problems
with $p(x)\equiv2^{\ast}$. Let us mention some of them.

In \cite{P65}, Pohozaev showed that the problem%
\begin{equation}
\left\{
\begin{array}
[c]{ll}%
-\Delta u=\lambda u+|u|^{2^{\ast}-2}u,\quad u>0 & \mathrm{in}\quad G,\\
u\in H_{0}^{1}(G), &
\end{array}
\right.  \tag{$P_1$}\label{int3}%
\end{equation}
does not admit a non-trivial solution if $\lambda\leq0$ and the bounded domain
$G$ is strictly star-shaped with respect to the origin in $\mathbb{R}^{N}$,
$N\geq3.$

In \cite{BN83}, Brezis and Nirenberg showed that if $N \geq 4$, the problem $(P_1)$ has a positive
solution for every $\lambda \in (0, \lambda_1)$, where $\lambda_1$ denotes the first eigenvalue of $(-\Delta, H_0^{1}(\Omega))$. If $N = 3$,
they proved that there exists $\lambda_* \in [0, \lambda_1)$ such that for any $\lambda \in (\lambda_*,\lambda_1)$ the problem $(P_1)$ 
admits a positive solution and, in the particular case where $G$ is a ball, a positive solution exists if, and only if, $\lambda \in (\lambda_1 /4, \lambda_1)$

Brezis and Nirenberg showed that if $G=B_{1}(0)$ and $N=3$,
then there exists $\lambda_{\ast}>0$ such that (\ref{int3}) does not have
solution for $\lambda\leq\lambda_{\ast}.$

In \cite{C84}, Coron proved that if there exist $R,r>0$ such that%
\[
G\supset\{x\in\mathbb{R}^{N}:r<\left\vert x\right\vert <R\}\quad
\mathrm{and}\quad\overline{G}\not \supset \{x\in\mathbb{R}^{N}:\left\vert
x\right\vert <r\}
\]
and $\frac{R}{r}$ is sufficiently large, then the problem (\ref{int3}) with
$\lambda=0$ has a positive solution in $H_{0}^{1}(G)$.

In \cite{BC88}, Bahri and Coron showed that if $\lambda=0$ and $\mathcal{H}%
_{i}(G;\mathbb{Z}/2)\neq0$ ($i$-th homology group) for some $i>0$, then the
problem (\ref{int3}) has at least one positive solution. The condition on the
homology group is valid, for example, if $\partial G$ is not connected.

Existence results for (\ref{int3}) related to the topology of $G$ were also
obtained by Bahri, in \cite{B89}. In
\cite{CCL92}, Carpio, Comte and Lewandowski obtained nonexistence results for
(\ref{int3}), with $\lambda=0,$ in contractible nonstarshaped domains.

On the other hand, the subcritical problem%
\begin{equation}
\left\{
\begin{array}
[c]{ll}%
-\Delta u=|u|^{q-2}u & \mathrm{in}\quad G,\\
u=0 & \mathrm{on}\quad\partial G,
\end{array}
\right.  \tag{$P_2$}\label{int5}%
\end{equation}
where $2<q<2^{\ast}$, has an unbounded set of solutions in $H_{0}^{1}(G)$ (See
\cite{GP87}). The problem (\ref{int5}) with $q=2^{\ast}-\epsilon$
($\epsilon>0$) was studied in the papers \cite{AP87} and \cite{GP92}. In
former, the authors considered $\Omega$ a ball and determined the exact
asymptotic behavior of the corresponding (radial) solutions $u_{\epsilon},$ as
$\epsilon\rightarrow0$. In \cite{GP92}, where a general bounded domain was
considered, the authors provided an alternative for the asymptotic behavior of
$u_{\epsilon},$ as $\epsilon\rightarrow0.$ More precisely, they showed that if
$\frac{1}{N}S^{\frac{N}{2}}<\lim_{\epsilon\rightarrow0}J_{\epsilon
}(u_{\epsilon})<\frac{2}{N}S^{\frac{N}{2}},$ then $u_{\epsilon}$ converges to
either a Dirac mass or a solution of critical problem (i.e (\ref{int5}) with
$q=p^{\ast}$). Here, $J_{\epsilon}$ denotes the energy functional associated
with the problem (\ref{int5}) and $q=2^{\ast}-\epsilon$, and
\begin{equation}
S:=\inf\left\{  \frac{\left\Vert u\right\Vert _{1,2}^{2}}{\left\Vert
u\right\Vert _{2^{\ast}}^{2}}:u\in D^{1,2}(\mathbb{R}^{N})\setminus\left\{
0\right\}  \right\}  \label{prb0}%
\end{equation}
is the Sobolev constant, which is given by the expression
\[
S:=\pi N(N-2)\left( \frac{\Gamma(N/2)}{\Gamma(N)}\right) ^{\frac{2}{N}},
\]
where $\Gamma(t)=\int_{0}^{\infty}s^{t-1}e^{-s}$\textrm{$d$}$s$ is the Gamma
Function (see Aubin \cite{Aubin} and Talenti\cite{T76}).

In \cite{KS08}, the compactness of the embedding $H_{0}^{1}(G)\hookrightarrow
L^{p(x)}(G),$ for a bounded domain $G$ and variable exponent $1\leq
p(x)\leq2^{\ast},$ was studied (for the definition and properties of
$L^{p(x)}(G)$ see \cite{FZ01}). It was showed the existence of a positive
solution of (\ref{int1}) under the hypothesis of existence of a point
$x_{0}\in G$, a small $\eta>0$, $0<l<1$ and $c_{0}>0$ such that $p(x_{0}%
)=2^{\ast}$ and%
\[
p(x)\leq2^{\ast}-\frac{c_{0}}{(\log(1/\left\vert x-x_{0}\right\vert ))^{l}%
},\quad\ \left\vert x-x_{0}\right\vert \leq\eta.
\]

When $G=\mathbb{R}^{N}$ and $p(x)\equiv2^{\star}$ the equation (\ref{int1})
becomes%
\begin{equation}
\left\{
\begin{array}
[c]{ll}%
-\Delta u=u^{2^{\ast}-1},\quad u>0 & \mathrm{in}\quad\mathbb{R}^{N}\\
u\in D^{1,2}(\mathbb{R}^{N}). &
\end{array}
\right.  \tag{$P_3$}\label{int4}%
\end{equation}
It is well known that the function
\[
w(x)=\frac{C_{N}}{(1+\left\vert x\right\vert ^{2})^{\frac{N-2}{2}}},\quad
C_{N}:=[N(N-2)]^{\frac{N-2}{4}},
\]
is a ground state solution of (\ref{int4}) and satisfies%
\[
\int_{\mathbb{R}^{N}}\left\vert \nabla w\right\vert ^{2}\mathrm{d}%
x=\int_{\mathbb{R}^{N}}\left\vert w\right\vert ^{2^{\ast}}\mathrm{d}%
x=S^{N/2}.
\]

In \cite{AS06} was studied the existence of nonnegative solutions of
$-\operatorname{div}(\left\vert \nabla u\right\vert ^{p(x)-2}\nabla
u)=u^{q(x)-1}$ in $\mathbb{R}^{N},$ where the variable exponents $p(x)$ and
$q(x)$ are radially symmetric functions satisfying $1<\operatorname*{essinf}%
_{\mathbb{R}^{N}}p(x)\leq\operatorname*{esssup}_{\mathbb{R}^{N}}p(x)<N$,
$p(x)\leq q(x)\leq2^{\ast}$ and
\[
p(x)=2,\quad q(x)=2^{\ast}\quad\mathrm{if\,either}\quad\left\vert x\right\vert
\leq\delta\quad\mathrm{or}\quad\left\vert x\right\vert \geq R,
\]
for constants $0<\delta<R.$

Finally, in \cite{LL15}, Liu, Liao and Tang proved the existence of a ground
state solution for (\ref{int1}) with $G=\mathbb{R}^{N}$ and $p$ given by
\[
p(x)=\left\{
\begin{array}
[c]{lll}%
p & \mathrm{if} & x\in\Omega\\
2^{\ast} & \mathrm{if} & x\in\mathbb{R}^{N}\setminus\Omega,
\end{array}
\right.
\]
where the constant $p$ belongs to $(2,2^{\ast})$ and $\Omega\subset
\mathbb{R}^{N}$ with nonempty interior.

In Section \ref{Sec3}, motivated by the results of \cite{LL15}, we use some properties of the 
Nehari manifold to show that the problem (\ref{int1}) has at least one
ground state solution when $G=\mathbb{R}^{N}$ and $p\in C(\mathbb{R}%
^{N},\mathbb{R})$ is a function satisfying the condition ($H_{1}$).

In Section \ref{Sec4}, we continue the study of (\ref{int1}), but assuming
that $G$ is a bounded domain in $\mathbb{R}^{N}$ and that the function $p\in
C(G,\mathbb{R}),$ satisfying \textrm{(}$\mathrm{H}_{1}$\textrm{)}, also verifies

\begin{description}
\item[($\mathrm{H}_{2}$)] There exists a subdomain $U$ of $\Omega$ such that
$S_{2}(U)<1$ and%
\[
p(x)\equiv q,\quad\forall\ x\in U,
\]
where: $\Omega$ is the set defined in \textrm{(}$\mathrm{H}_{1}$\textrm{)},
$p^{-}\leq q<\min\{\bar{q},p^{+}\}$ for some $\bar{q}\in(2,2^{\ast}]$ (which
will be defined later) and $S_{2}(U)$ is the best constant of the embedding
$H_{0}^{1}(U)\hookrightarrow L^{2}(U)$.
\end{description}

Under such conditions, we show that the problem (\ref{int1}) has at least one
ground state solution and also present sufficient conditions for $S_{2}(U)<1$
to hold, when the subdomain $U$ is either a ball $B_{R}$ or an annular-shaped
domain $B_{R}\setminus\overline{B_{r}}$, with $\overline{B_{r}}\subset B_{R}.$
Moreover, we show that if $R$ and $R-r$ are sufficiently large, then
$S_{2}(U)<1$ for $U=B_{R}$ and $U=B_{R}\setminus\overline{B_{r}}$, respectively.

We believe it is possible to find further conditions that assure the existence
of at least one solution for (\ref{int1}) in the case where $G$ is a bounded
domain, and we hope to return to this subject in the future.

\section{The semilinear elliptic problem in $\mathbb{R}^{N}$\label{Sec3}}

In this section, we consider the semilinear elliptic problem with variable
exponent
\begin{equation}
\left\{
\begin{array}
[c]{ll}%
-\Delta u=u^{p(x)-1},\quad u>0 & \mathrm{in}\quad\mathbb{R}^{N}\\
u\in D^{1,2}(\mathbb{R}^{N}), &
\end{array}
\right.  \label{prb1}%
\end{equation}
where $N\geq3$ and $p:\mathbb{R}^{N}\rightarrow\mathbb{R}$ is a continuous
function verifying the hypothesis $\mathrm{(H}_{1}\mathrm{)}$.

We recall that the space $D^{1,2}(\mathbb{R}^{N})$ is the completion of
$C_{0}^{\infty}(\mathbb{R}^{N})$ with respect to the norm%

\[
\left\Vert u\right\Vert _{1,2}:=\left(  \int_{\mathbb{R}^{N}}|\nabla
u|^{2}\,\mathrm{d}x\right)  ^{\frac{1}{2}}.
\]
The dual space of $D^{1,2}(\mathbb{R}^{N})$ will be denoted by $D^{-1}$.

The energy functional $I:D^{1,2}(\mathbb{R}^{N})\rightarrow\mathbb{R}$
associated with (\ref{prb1}) is given by%
\[
I(u)=\frac{1}{2}\left\Vert u\right\Vert _{1,2}^{2}-\int_{\mathbb{R}^{N}}%
\frac{1}{p(x)}(u^{+})^{p(x)}\,\mathrm{d}x,
\]
where $u^{+}(x)=\max\{u(x),0\}$ and $u^{-}(x)=\min\{u(x),0\}$. Hence, under
the hypothesis \textrm{(}$\mathrm{H}_{1}$\textrm{)}, we can write
\[
I(u)=\frac{1}{2}\left\Vert u\right\Vert _{1,2}^{2}-\int_{\Omega_{\delta}}%
\frac{1}{p(x)}(u^{+})^{p(x)}\,\mathrm{d}x-\frac{1}{2^{\ast}}\int
_{\mathbb{R}^{N}\setminus\Omega_{\delta}}(u^{+})^{2^{\ast}}\,\mathrm{d}x.
\]

For a posterior use, let us estimate the second term in the above expression.
For this, let $u\in D^{1,2}(\mathbb{R}^{N})$ and consider the set
$E=\{x\in\Omega_{\delta}:\left\vert u(x)\right\vert <1\}.$ Then,
\begin{align*}
\int_{\Omega_{\delta}}\frac{1}{p(x)}(u^{+})^{p(x)}\,\mathrm{d}x  &  \leq
\frac{1}{p^{-}}\int_{E}(u^{+})^{p^{-}}\,\mathrm{d}x+\frac{1}{p^{-}}%
\int_{\Omega_{\delta}\setminus E}(u^{+})^{2^{\ast}}\,\mathrm{d}x\\
&  \leq\frac{1}{p^{-}}\int_{\Omega_{\delta}}\left\vert u\right\vert ^{p^{-}%
}\,\mathrm{d}x+\frac{1}{p^{-}}\left\Vert u\right\Vert _{2^{\ast}}^{2^{\ast}}\\
&  \leq\frac{1}{p^{-}}\left(  \int_{\Omega_{\delta}}\left\vert u\right\vert
^{2^{\ast}}\,\mathrm{d}x\right)  ^{\frac{p^{-}}{2^{\ast}}}\left\vert
\Omega_{\delta}\right\vert ^{\frac{2^{\ast}-p^{-}}{2^{\ast}}}+\frac{1}{p^{-}%
}\left\Vert u\right\Vert _{2^{\ast}}^{2^{\ast}}\\
&  \leq\frac{1}{p^{-}}\left\vert \Omega_{\delta}\right\vert ^{\frac{2^{\ast
}-p^{-}}{2^{\ast}}}\left\Vert u\right\Vert _{2^{\ast}}^{p^{-}}+\frac{1}{p^{-}%
}\left\Vert u\right\Vert _{2^{\ast}}^{2^{\ast}},
\end{align*}
where\ we have used ($\mathrm{H}_{1}$) and H\"{o}lder's inequality. Hence, it
follows from (\ref{prb0}) and ($\mathrm{H}_{1}c$) that%
\begin{equation}
\frac{1}{2^{\ast}}\int_{\Omega_{\delta}}(u^{+})^{p(x)}\,\mathrm{d}x\leq
\int_{\Omega_{\delta}}\frac{1}{p(x)}(u^{+})^{p(x)}\,\mathrm{d}x\leq
a\left\Vert u\right\Vert _{1,2}^{p^{-}}+b\left\Vert u\right\Vert
_{1,2}^{2^{\ast}}, \label{aux1}%
\end{equation}
where%
\begin{equation}
a=\frac{1}{p^{-}}\left\vert \Omega_{\delta}\right\vert ^{\frac{2^{\ast}-p^{-}%
}{2^{\ast}}}\quad\mathrm{and}\quad b=\frac{S^{-\frac{2^{\ast}}{2}}}{p^{-}}.
\label{aux2}%
\end{equation}

We observe from (\ref{aux1}) that the functional $I$ is well defined.

The next lemma establishes that $I$ is of class $C^{1}$. Since its proof is
standard, it will be omitted.

\begin{lemma}
\label{lemb1} Let $p\in C(\mathbb{R}^{N},\mathbb{R})$ a function satisfying
\textrm{(H}$_{1}a$\textrm{)}. Then $I\in C^{1}(D^{1,2}(\mathbb{R}%
^{N}),\mathbb{R})$ and%
\begin{equation}
I^{\prime}(u)(v)=\int_{\mathbb{R}^{N}}\nabla u\cdot\nabla v\,\mathrm{d}%
x-\int_{\mathbb{R}^{N}}(u^{+})^{p(x)-1}v\,\mathrm{d}x,\quad\forall\ u,v\in
D^{1,2}(\mathbb{R}^{N}). \label{prb4}%
\end{equation}

\end{lemma}

The previous lemma ensures that $u\in D^{1,2}(\mathbb{R}^{N})$ is a solution
of (\ref{prb1}) if, and only if, $u$ is a critical point of $I$ (i.e.
$I^{\prime}(u)=0$). We would like to point out that critical points $u$ of $I$
are nonnegative, since
\[
0=I^{\prime}(u)(u^{-})=\int_{\mathbb{R}^{N}}\nabla u\cdot\nabla u^{-}%
\,\mathrm{d}x-\int_{\mathbb{R}^{N}}(u^{+})^{p(x)-1}u^{-}\,\mathrm{d}%
x=\left\Vert u^{-}\right\Vert _{1,2}^{2}.
\]
We would also emphasize the following consequence of the strongly maximum
principle: if $u$ is a nontrivial critical point of $I,$ then $u$ is a
positive solution of (\ref{prb1}).

\subsection{The Nehari manifold}

In this subsection we prove some properties of the Nehari manifold associated
with (\ref{prb1}), which is defined by%
\[
\mathcal{N}:=\{u\in D^{1,2}(\mathbb{R}^{N})\setminus\{0\}:J(u)=0\},
\]
where
\[
J(u):=I^{\prime}(u)(u)=\left\Vert u\right\Vert _{1,2}^{2}-\int_{\mathbb{R}%
^{N}}(u^{+})^{p(x)}\,\mathrm{d}x.
\]

Of course, critical points of $I$ belong to $\mathcal{N}.$

\begin{definition}
	\label{GS}We say that $u\in\mathcal{N}$ is a \textit{ground state solution}
	for (\ref{prb1}) if $I^{\prime}(u)=0$ and $I(u)=m.$
\end{definition}

Next, we will show important properties involving the Nehari manifold, which are crucial in our approach.

\begin{proposition}
\label{lemb2} Assume that ($\mathrm{H}_{1}$) holds. Then%
\[
m:=\inf_{u\in\mathcal{N}}I(u)>0.
\]

\end{proposition}

\begin{proof}
For an arbitrary $u\in\mathcal{N}$ we have
\[
\left\Vert u\right\Vert _{1,2}^{2}=\int_{\mathbb{R}^{N}}(u^{+})^{p(x)}%
\,\mathrm{d}x=\int_{\Omega_{\delta}}(u^{+})^{p(x)}\,\mathrm{d}x+\int
_{\mathbb{R}^{N}\setminus\Omega_{\delta}}(u^{+})^{p(x)}\,\mathrm{d}x\leq
\int_{\Omega_{\delta}}(u^{+})^{p(x)}\,\mathrm{d}x+\left\Vert u\right\Vert
_{1,2}^{2^{\ast}}.
\]
Thus, it follows from (\ref{aux1}) that
\[
\left\Vert u\right\Vert _{1,2}^{2}\leq C_{1}\left\Vert u\right\Vert
_{1,2}^{p^{-}}+C_{2}\left\Vert u\right\Vert _{1,2}^{2^{\ast}},
\]
where $C_{1}$ and $C_{2}$ denote positive constants that do not depend on $u$.
Consequently,
\[
1\leq C_{1}\left\Vert u\right\Vert _{1,2}^{p^{-}-2}+C_{2}\left\Vert
u\right\Vert _{1,2}^{2^{\ast}-2},
\]
from which conclude that there exists $\eta>0$ such that
\begin{equation}
\left\Vert u\right\Vert _{1,2}\geq\eta,\quad\forall u\in\mathcal{N}.
\label{aux4}%
\end{equation}
Therefore,
\begin{align}
I(u)  &  =I(u)-\frac{1}{p^{-}}I^{\prime}(u)(u)\label{aux3}\\
&  =\left(  \frac{1}{2}-\frac{1}{p^{-}}\right)  \left\Vert u\right\Vert
_{1,2}^{2}+\int_{\mathbb{R}^{N}}\left(  \frac{1}{p^{-}}-\frac{1}{p(x)}\right)
(u^{+})^{p(x)}\mathrm{d}x\geq\left(  \frac{1}{2}-\frac{1}{p^{-}}\right)
\left\Vert u\right\Vert _{1,2}^{2}.\nonumber
\end{align}
In view of (\ref{aux4}), this implies that $m\geq\left(  \frac{1}{2}-\frac
{1}{p^{-}}\right)  \eta^{2}>0$.
\end{proof}

\begin{proposition}
\label{lemb3} Assume ($\mathrm{H}_{1}$). Then, for each $u\in D^{1,2}%
(\mathbb{R}^{N})$ with $u^{+}\not \equiv 0$, there exists a unique $t_{u}>0$
such that $t_{u}u\in\mathcal{N}$.
\end{proposition}

\begin{proof}
Let
\[
f(t):=I(tu)=\frac{t^{2}}{2}\left\Vert u\right\Vert _{1,2}^{2}-\int
_{\mathbb{R}^{N}}\frac{t^{p(x)}}{p(x)}(u^{+})^{p(x)}\mathrm{d}x,\quad
t\in(0,+\infty).
\]
We note that
\[
f^{\prime}(t)=I^{\prime}(tu)(u)=t\left\Vert u\right\Vert _{1,2}^{2}%
-\int_{\mathbb{R}^{N}}t^{p(x)-1}(u^{+})^{p(x)}\mathrm{d}x=\frac{1}%
{t}J(tu),\quad\forall\,t\in(0,+\infty).
\]

Since $1<p^{-}-1\leq p(x)-1$ we have%
\[
f^{\prime}(t)\geq t\left(  \left\Vert u\right\Vert _{1,2}^{2}-t^{p^{-}-2}%
\int_{\mathbb{R}^{N}}(u^{+})^{p(x)}\mathrm{d}x\right)  ,\quad\forall
\,t\in(0,1)
\]
and
\[
f^{\prime}(t)\leq t\left(  \left\Vert u\right\Vert _{1,2}^{2}-t^{p^{-}-2}%
\int_{\mathbb{R}^{N}}(u^{+})^{p(x)}\mathrm{d}x\right)  ,\quad\forall\,t\geq1.
\]
Thus, we can see that $f^{\prime}(t)>0$ for all $t>0$ sufficiently small and
also that $f^{\prime}(t)<0$ for all $t\geq1$ sufficiently large. Therefore,
there exists $t_{u}>0$ such that
\[
f^{\prime}(t_{u})=\frac{1}{t_{u}}J(t_{u}u)=0,
\]
showing that $t_{u}u\in\mathcal{N}$.

In order to prove the uniqueness of $t_{u},$ let us assume that $0<t_{1}%
<t_{2}$ satisfy $f^{\prime}(t_{1})=f^{\prime}(t_{2})=0$. Then,
\[
\left\Vert u\right\Vert _{1,2}^{2}=\int_{\mathbb{R}^{N}}t_{1}^{p(x)-2}%
(u^{+})^{p(x)}\mathrm{d}x=\int_{\mathbb{R}^{N}}t_{2}^{p(x)-2}(u^{+}%
)^{p(x)}\mathrm{d}x.
\]
Hence,
\[
\int_{\mathbb{R}^{N}}(t_{1}^{p(x)-2}-t_{2}^{p(x)-2})(u^{+})^{p(x)}%
\mathrm{d}x=0.
\]
Since $t_{1}^{p(x)-2}>t_{2}^{p(x)-2}$ for all $x\in\mathbb{R}^{N}$, the above
equality leads to the contradiction: $u^{+}\equiv0.$
\end{proof}

\begin{proposition}
\label{lemb4} Assume that ($\mathrm{H}_{1}$) holds. Then,
\[
J^{\prime}(u)(u)\leq(2-p^{-})\eta^{2}<0,\quad\forall\,u\in\mathcal{N},
\]
where $\eta$ was given in (\ref{aux4}). Hence, $J^{\prime}(u)\neq0$ for all $u\in\mathcal{N}$.
\end{proposition}

\begin{proof}
For $u\in\mathcal{N}$ we have
\begin{align*}
J^{\prime}(u)(u)  &  =2\left\Vert u\right\Vert _{1,2}^{2}-\int_{\mathbb{R}%
^{N}}p(x)(u^{+})^{p(x)}\,\mathrm{d}x\\
&  \leq2\left\Vert u\right\Vert _{1,2}^{2}-p^{-}\int_{\mathbb{R}^{N}}%
(u^{+})^{p(x)}\,\mathrm{d}x=(2-p^{-})\left\Vert u\right\Vert _{1,2}^{2}%
\leq(2-p^{-})\eta^{2}<0,
\end{align*}
according to (\ref{aux4}).  
\end{proof}

\begin{proposition}
\label{lemb5} Assume ($\mathrm{H}_{1}$) and that there exists $u_{0}%
\in\mathcal{N}$ such that $I(u_{0})=m$. Then $u_{0}$ is ground state solution
for (\ref{prb1}).
\end{proposition}

\begin{proof}
Since $m$ is the minimum of $I$ on $\mathcal{N},$ Lagrange multiplier theorem
implies that there exists $\lambda\in\mathbb{R}$ such that $I^{\prime}%
(u_{0})=\lambda J^{\prime}(u_{0})$. Thus
\[
\lambda J^{\prime}(u_{0})(u_{0})=I^{\prime}(u_{0})(u_{0})=J(u_{0})=0.
\]
According to the previous lemma, $\lambda=0$, and so, $I^{\prime}(u_{0})=0.$
\end{proof}

The next lemma shows that, under ($\mathrm{H}_{1}$), there exists a
Palais-Smale sequence for $I$ associated with the minimum $m.$

\begin{proposition}
\label{lemb6} Assume ($\mathrm{H}_{1}$). There exists a sequence
$(u_{n})\subset\mathcal{N}$ such that: $u_{n}\geq0$ in $\mathbb{R}^N,$
$I(u_{n})\rightarrow m,$ and $I^{\prime}(u_{n})\rightarrow0$ in $D^{-1}.$
\end{proposition}

\begin{proof}
According to the Ekeland variational principle (see \cite[Theorem 8.5]{W96}),
there exist $(u_{n})\subset\mathcal{N}$ and $(\lambda_{n})\subset\mathbb{R}$
such that
\[
I(u_{n})\rightarrow m\quad\mathrm{and}\quad I^{\prime}(u_{n})-\lambda
_{n}J^{\prime}(u_{n})\rightarrow0\quad\mathrm{in}\quad D^{-1}.
\]
It follows from (\ref{aux3}) that
\[
\left(  \frac{1}{2}-\frac{1}{p^{-}}\right)  \Vert u_{n}\Vert_{1,2}^{2}\leq
I(u_{n}).
\]
This implies that $(u_{n})$ is bounded in $D^{1,2}(\mathbb{R}^{N}).$ Hence,
taking into account that
\[
\left\vert I^{\prime}(u_{n})(u_{n})-\lambda_{n}J^{\prime}(u_{n})(u_{n}%
)\right\vert \leq\left\Vert I^{\prime}(u_{n})-\lambda_{n}J^{\prime}%
(u_{n})\right\Vert _{D^{-1}}\left\Vert u_{n}\right\Vert _{1,2}%
\]
we have
\[
I^{\prime}(u_{n})(u_{n})-\lambda_{n}J^{\prime}(u_{n})(u_{n})\rightarrow0.
\]
Using the fact that $I^{\prime}(u_{n})(u_{n})=0$ we then conclude, from
Proposition \ref{lemb4}, that $\lambda_{n}\rightarrow0.$ Consequently,
$I^{\prime}(u_{n})\rightarrow0$ in $D^{-1}.$

In order to complete this proof, let us show that the sequence $(u_{n}^{+})$
satisfies $I(u_{n}^{+})\rightarrow m$ and $I^{\prime}(u_{n}^{+})\rightarrow0$
in $D^{-1}.$ Indeed, since $\left\Vert u_{n}^{-}\right\Vert _{1,2}=I^{\prime
}(u_{n})(u_{n}^{-})\rightarrow0,$ we derive%
\[
I(u_{n}^{+})=I(u_{n})-\frac{1}{2}\left\Vert u_{n}^{-}\right\Vert _{1,2}%
^{2}\rightarrow m.
\]
Moreover, 
\[
\sup_{\|\phi\| \leq 1}\left\vert I^{\prime}(u_{n}^{+})(\phi)\right\vert =\sup_{\|\phi\| \leq 1}\left\vert I^{\prime}%
(u_{n})(\phi)-\int_{\mathbb{R}^{N}}\nabla(u_{n}^{-})\cdot\nabla\phi
\,\mathrm{d}x\right\vert \leq\left\Vert I^{\prime}(u_{n})\right\Vert _{D^{-1}%
}+\left\Vert u_{n}^{-}\right\Vert _{1,2}\rightarrow0.
\]
\end{proof}

The next lemma establishes an estimate from above for $m$.

\begin{proposition}
\label{lemb7} Assume ($\mathrm{H}_{1}$). Then $m<\frac{1}{N}S^{\frac{N}{2}}$, where $S$ denotes the Sobolev constant defined by (\ref{prb0})
\end{proposition}

\begin{proof}
We define
\[
w_{k}(x):=w(x+ke_{N}),\quad e_{N}=(0,0,...,0,1),
\]
where $w:\mathbb{R}^{N}\rightarrow\mathbb{R}$ is the ground state solution of
(\ref{int4}) given by
\[
w(x)=\frac{[N(N-2)]^{\frac{N-2}{4}}}{(1+\left\vert x\right\vert ^{2}%
)^{\frac{N-2}{2}}},
\]
which satisfies
\begin{equation}
\left\Vert w\right\Vert _{1,2}^{2}=\left\Vert w\right\Vert _{2^{\ast}%
}^{2^{\ast}}=S^{\frac{N}{2}}. \label{prb8}%
\end{equation}

A direct computation shows that $\left\Vert w_{k}\right\Vert _{2^{\ast}%
}=\left\Vert w\right\Vert _{2^{\ast}}$ and $\left\Vert w_{k}\right\Vert
_{1,2}=\left\Vert w\right\Vert _{1,2}.$ Moreover, exploring the expression of
$w$ we can easily check that $w_{k}\rightarrow0$ uniformly in bounded sets
and, therefore,
\begin{equation}
\lim_{k\rightarrow\infty}%
%TCIMACRO{\dint _{\Omega_{\delta}}}%
%BeginExpansion
{\displaystyle\int_{\Omega_{\delta}}}
%EndExpansion
\left\vert w_{k}\right\vert ^{\alpha}\mathrm{d}x=0, \label{aux5}%
\end{equation}
for any $\alpha>0.$

By Proposition \ref{lemb3}, there exists $t_{k}>0$ such that $t_{k}w_{k}%
\in\mathcal{N},$ which means that
\begin{equation}
t_{k}^{2}\left\Vert w_{k}\right\Vert _{1,2}^{2}=\int_{\Omega_{\delta}}%
(t_{k}w_{k})^{p(x)}\mathrm{d}x+\int_{\mathbb{R}^{N}\setminus\Omega_{\delta}%
}(t_{k}w_{k})^{2^{\ast}}\mathrm{d}x. \label{prb9}%
\end{equation}
Hence,
\[
\left\Vert w\right\Vert _{1,2}^{2}=\left\Vert w_{k}\right\Vert _{1,2}^{2}\geq
t_{k}^{2^{\ast}-2}\int_{\mathbb{R}^{N}\setminus\Omega_{\delta}}\left\vert
w_{k}\right\vert ^{2^{\ast}}\mathrm{d}x
\]
and then, by using (\ref{aux5}) for $\alpha=2^{\ast},$ we can verify that the
sequence $\left(  t_{k}\right)  $ is bounded:
\[
\limsup_{k\rightarrow\infty}t_{k}\leq\limsup_{k\rightarrow\infty}\left(
\frac{\left\Vert w\right\Vert _{1,2}^{2}}{\int_{\mathbb{R}^{N}\setminus
\Omega_{\delta}}(w_{k})^{2^{\ast}}\mathrm{d}x}\right)  ^{\frac{1}{2^{\ast}-2}%
}=\left(  \frac{\left\Vert w\right\Vert _{1,2}^{2}}{\left\Vert w\right\Vert
_{2^{\ast}}^{2^{\ast}}}\right)  ^{\frac{1}{2^{\ast}-2}}=1.
\]
Moreover, since $t_{k}w_{k}\in\mathcal{N}$, 
\begin{align*}
m  &  \leq I(t_{k}w_{k})\\
&  =\frac{t_{k}^{2}}{2}\left\Vert w_{k}\right\Vert _{1,2}^{2}-\int
_{\mathbb{R}^{N}\setminus\Omega_{\delta}}\frac{(t_{k}w_{k})^{2^{\ast}}%
}{2^{\ast}}\,\mathrm{d}x-\int_{\Omega_{\delta}}\frac{(t_{k}w_{k})^{p(x)}%
}{p(x)}\,\mathrm{d}x\\
&  =\frac{t_{k}^{2}}{2}S^{N/2}-\int_{\mathbb{R}^{N}}\frac{(t_{k}%
w_{k})^{2^{\ast}}}{2^{\ast}}\,\mathrm{d}x+\int_{\Omega_{\delta}}\frac
{(t_{k}w_{k})^{2^{\ast}}}{2^{\ast}}\,\mathrm{d}x-\int_{\Omega_{\delta}}%
\frac{(t_{k}w_{k})^{p(x)}}{p(x)}\,\mathrm{d}x\\
&  =S^{N/2}\left(  \frac{t_{k}^{2}}{2}-\frac{t_{k}^{2^{\ast}}}{2^{\ast}%
}\right)  +\int_{\Omega_{\delta}}\left(  \frac{(t_{k}w_{k})^{2^{\ast}}%
}{2^{\ast}}-\frac{(t_{k}w_{k})^{p(x)}}{p(x)}\right)  \,\mathrm{d}x\\
&  \leq\frac{S^{N/2}}{N}+\int_{\Omega_{\delta}}\left(  \frac{(t_{k}%
w_{k})^{2^{\ast}}}{2^{\ast}}-\frac{(t_{k}w_{k})^{p(x)}}{p(x)}\right)
\,\mathrm{d}x,
\end{align*}
where we have used that the maximum of the function $t\in\lbrack
0,\infty)\longmapsto\frac{t^{2}}{2}-\frac{t^{2^{\ast}}}{2^{\ast}}$ is
$\frac{1}{N}$.

Combining the boundedness of the sequence $\left(  t_{k}\right)  $ with the
fact that $w_{k}\rightarrow0$ uniformly in $\Omega_{\delta},$ we can select
$k$ sufficiently large, such that $t_{k}w_{k}\leq1$ in $\Omega_{\delta}.$
Therefore, for this $k$, 
\begin{align*}
m  &  \leq\frac{S^{N/2}}{N}+\int_{\Omega_{\delta}}\left(  \frac{(t_{k}%
w_{k})^{2^{\ast}}}{2^{\ast}}-\frac{(t_{k}w_{k})^{2^{\ast}}}{p(x)}\right)
\,\mathrm{d}x\\
&  =\frac{S^{N/2}}{N}+t_{k}^{2^{\ast}}\int_{\Omega_{\delta}}(w_{k})^{2^{\ast}%
}\left(  \frac{1}{2^{\ast}}-\frac{1}{p(x)}\right)  \,\mathrm{d}x<\frac
{S^{N/2}}{N},
\end{align*}
since the latter integrand is strictly positive in $\Omega$ with has positive
$N$-dimensional Lebesgue measure.
\end{proof}

\subsection{Existence of a ground state solution.}

Our main result in this section is the following.

\begin{theorem}
\label{lemb8} Assume that \textrm{(}$\mathrm{H}_{1}$\textrm{)} holds. Then,
the problem (\ref{prb1}) has at least one positive ground state solution.
\end{theorem}

We prove this theorem throughout this subsection by using the following
well-known result.

\begin{lemma}
[Lions' Lemma]Let $(u_{n})$ be a sequence in $D^{1,2}(\mathbb{R}^{N}),$ $N>2,$ satisfying

\begin{itemize}
\item $u_{n}\rightharpoonup u\quad\mathrm{in}\,D^{1,2}(\mathbb{R}^{N});$

\item $\left\vert \nabla u_{n}\right\vert ^{2}\rightharpoonup\mu
\quad\mathrm{in}\,\mathcal{M}(\mathbb{R}^{N});$

\item $\left\vert u_{n}\right\vert ^{2^{\ast}}\rightharpoonup\nu
\quad\mathrm{in}\,\mathcal{M}(\mathbb{R}^{N}).$
\end{itemize}

Then, there exist an at most enumerable set of indices $J,$ points
$(x_{i})_{i\in J}$ and positive numbers $(\nu_{i})_{i\in J}$ such that

\begin{description}
\item[i)] $\nu=\left\vert u\right\vert ^{2^{\ast}}+\sum_{i\in J}\nu_{i}%
\delta_{x_{i}},$

\item[ii)] $\mu(\{x_{i}\})\geq\nu_{i}^{2/2^{\ast}}S$, $\quad\mathrm{for}%
\,$\textrm{any} $i\in J$,
\end{description}
where $\delta_{x_{i}}$ denotes the Dirac measure supported at $x_{i}$.
\end{lemma}

We know from Proposition \ref{lemb6} that there exists a sequence
$(u_{n})\subset\mathcal{N}$ satisfying: $u_{n}\geq0$ in $\ \mathbb{R}^{N},$
$I(u_{n})\rightarrow m,$ and $I^{\prime}(u_{n})\rightarrow0$ in $D^{-1}$.
Since $(u_{n})$ is bounded in $D^{1,2}(\mathbb{R}^{N})$, we can assume (by
passing to a subsequence) that there exists $u\in D^{1,2}(\mathbb{R}^{N})$
such that $u_{n}\rightharpoonup u$ in $D^{1,2}(\mathbb{R}^{N})$,
$u_{n}\rightarrow u$ in $L_{loc}^{s}(\mathbb{R}^{N})$ for $1\leq s<2^{\ast}$
and $u_{n}(x)\rightarrow u(x)$ a.e. in $\mathbb{R}^{N}$. Moreover, $\left\vert
\nabla u_{n}\right\vert ^{2}\rightharpoonup\mu$ and $\left\vert u_{n}%
\right\vert ^{2^{\ast}}\rightharpoonup\nu$ in $\mathcal{M}(\mathbb{R}^{N}).$

We claim that $u\not \equiv 0$. Indeed, let us suppose, by contradiction, that
$u\equiv0$. We affirm that this assumption implies that the set $J$ given by
Lions' lemma is empty. Otherwise, let us fix $i\in J,$ $x_{i}\in\mathbb{R}%
^{N},$ and $\nu_{i}>0$ as in the Lions' Lemma. Let $\varphi\in C_{c}^{\infty
}(\mathbb{R}^{N})$ such that%
\[
\varphi(x)=\left\{
\begin{array}
[c]{cc}%
1, & x\in B_{1}(0)\\
0, & x\notin B_{2}(0)
\end{array}
\right.
\]
and $0\leq\varphi(x)\leq1$ for all $x\in\mathbb{R}^{N}$, where $B_{1}$ and
$B_{2}$ denotes the balls centered at the origin, with radius $1$ and $2,$ respectively.

For $\epsilon>0$ fixed, define
\[
\varphi_{\epsilon}(x)=\varphi\left(  \frac{x-x_{i}}{\epsilon}\right)  .
\]
Since $(u_{n})$ is bounded in $D^{1,2}(\mathbb{R}^{N})$, the same holds for
the sequence $(\varphi_{\epsilon}u_{n}).$ Thus,
\[
\left\vert I^{\prime}(u_{n})(\varphi_{\epsilon}u_{n})\right\vert
\leq\left\Vert I^{\prime}(u_{n})\right\Vert _{D^{-1}}\left\Vert \varphi
_{\epsilon}u_{n}\right\Vert _{1,2}=o_{n}(1),
\]
so that
\[
\int_{\mathbb{R}^{N}}\nabla u_{n}\cdot\nabla(\varphi_{\epsilon}u_{n}%
)\,\mathrm{d}x=\int_{\mathbb{R}^{N}}(u_{n})^{p(x)}\varphi_{\epsilon
}\,\mathrm{d}x+o_{n}(1).
\]
Consequently,
\begin{equation}
\int_{\mathbb{R}^{N}}\varphi_{\epsilon}|\nabla u_{n}|^{2}\,\mathrm{d}%
x+\int_{\mathbb{R}^{N}}u_{n}\nabla u_{n}\cdot\nabla\varphi_{\epsilon
}\,\mathrm{d}x\leq\int_{\mathbb{R}^{N}}|u_{n}|^{p^{-}}\varphi_{\epsilon
}\,\mathrm{d}x+\int_{\mathbb{R}^{N}}|u_{n}|^{2^{\ast}}\varphi_{\epsilon
}\,\mathrm{d}x+o_{n}(1). \label{prb13}%
\end{equation}

According to Lions' lemma
\[
\int_{\mathbb{R}^{N}}|\nabla u_{n}|^{2}\varphi_{\epsilon}\,\mathrm{d}%
x\rightarrow\int_{\mathbb{R}^{N}}\varphi_{\epsilon}\,\mathrm{d}\mu
\quad\mathrm{and}\quad\int_{\mathbb{R}^{N}}|u_{n}|^{2^{\ast}}\varphi
_{\epsilon}\,\mathrm{d}x\rightarrow\int_{\mathbb{R}^{N}}\varphi_{\epsilon
}\,\mathrm{d}\nu.
\]

Since
\[
\left\vert \int_{\mathbb{R}^{N}}u_{n}\nabla u_{n}\cdot\nabla\varphi_{\epsilon
}\,\mathrm{d}x\right\vert \leq\left\Vert \nabla\varphi_{\epsilon}\right\Vert
_{\infty} \left(\int_{B_{2 \epsilon}(x_i)}|u_n|^{2}\,dx \right)^{\frac{1}{2}} \left\Vert u_{n}\right\Vert
_{1,2}\rightarrow0
\]
and
\[
\int_{\mathbb{R}^{N}}|u_{n}|^{p^{-}}\varphi_{\epsilon}\,\mathrm{d}%
x\rightarrow0
\]
it follows from (\ref{prb13}) that%
\begin{equation}
\int_{\mathbb{R}^{N}}\varphi_{\epsilon}\,\mathrm{d}\mu\leq\int_{\mathbb{R}%
^{N}}\varphi_{\epsilon}\,\mathrm{d}\nu,\quad\forall\,\epsilon>0. \label{prb14}%
\end{equation}
Now, making $\epsilon\rightarrow0$, we get
\[
\mu(\{x_{i}\})\leq\nu_{i}.
\]
Combining this inequality with the part \textrm{ii)} of Lions' lemma, we
obtain $\nu_{i}\geq S^{N/2}.$ It follows that
\[
S^{N/2}\leq S\nu_{i}^{2/2^{\ast}}\leq\mu(\{x_{i}\})\leq\nu_{i}.
\]

Let $\varphi\in C_{c}^{\infty}(\mathbb{R}^{N})$ such that $\varphi(x_{i})=1$
and $0\leq\varphi(x)\leq1$, for any $x\in\mathbb{R}^{N}$. Recalling that%
\[
I(u_{n})=I(u_{n})-\frac{1}{2^{\ast}}I^{\prime}(u_{n})(u_{n})=\frac{1}%
{N}\left\Vert u_{n}\right\Vert _{1,2}^{2}+\int_{\Omega_{\delta}}\left(
\frac{1}{2^{\ast}}-\frac{1}{p(x)}\right)  |u_{n}|^{p(x)}\mathrm{d}x,
\]
we have
\begin{equation}
I(u_{n})\geq\frac{1}{N}\int_{\mathbb{R}^{N}}\left\vert \nabla u_{n}\right\vert
^{2}\varphi\,\mathrm{d}x+\int_{\Omega_{\delta}}\left(  \frac{1}{2^{\ast}%
}-\frac{1}{p(x)}\right)  |u_{n}|^{p(x)}\mathrm{d}x. \label{prb15}%
\end{equation}

Since $p:\mathbb{R}\rightarrow\mathbb{R}$ is continuous, for each $\epsilon
>0$, there exists $\Omega_{\delta,\epsilon}\subset\Omega_{\delta}$ be such
that
\[
\left\vert \frac{1}{2^{\ast}}-\frac{1}{p(x)}\right\vert <\frac{\epsilon}%
{2M},\quad x\in\Omega_{\delta}\setminus\Omega_{\delta,\epsilon},
\]
where $M=\sup\limits_{n\in\mathbb{N}}\left(  \int_{\Omega_{\delta}}%
|u_{n}|^{p^{-}}+|u_{n}|^{2^{\ast}}\mathrm{d}x\right)  $. Thus,%
\begin{align*}
\left\vert \int_{\Omega_{\delta}}\left(  \frac{1}{2^{\ast}}-\frac{1}%
{p(x)}\right) |u_{n}|^{p(x)}\mathrm{d}x\right\vert  &  \leq\frac{\epsilon}%
{2M}\int_{\Omega_{\delta}\setminus\Omega_{\delta,\epsilon}}|u_{n}|%
^{p(x)}\mathrm{d}x+\left(  \frac{1}{p^{-}}-\frac{1}{2^{\ast}}\right)
\int_{\Omega_{\delta,\epsilon}}|u_{n}|^{p(x)}dx\\
&  \leq\frac{\epsilon}{2M}\int_{\Omega_{\delta}}(|u_{n}|
^{p^{-}}+\left\vert u_{n}\right\vert ^{2^{\ast}})dx+\left(  \frac{1}{p^{-}%
}-\frac{1}{2^{\ast}}\right)  \int_{\Omega_{\delta,\epsilon}}(\left\vert
u_{n}\right\vert ^{p^{-}}+\left\vert u_{n}\right\vert ^{q})dx\\
&  \leq\frac{\epsilon}{2}+\left(  \frac{1}{p^{-}}-\frac{1}{2^{\ast}}\right)
\int_{\Omega_{\delta}}(\left\vert u_{n}\right\vert ^{p^{-}}+\left\vert
u_{n}\right\vert ^{q})dx,
\end{align*}
where $2<p^{-}\leq p(x)\leq q<2^{\ast}$, for $x\in\Omega_{\delta,\epsilon}$.
Then, since $u_{n}\rightarrow0$ in $L_{loc}^{s}(\mathbb{R}^{N})$, for
$s\in\lbrack1,2^{\ast})$, and $\epsilon$ is arbitrary, we conclude that
\[
\lim_{n\rightarrow\infty}\int_{\Omega_{\delta}}\left(  \frac{1}{2^{\ast}%
}-\frac{1}{p(x)}\right) |u_{n}|^{p(x)}\mathrm{d}x=0.
\]
Therefore, by making $n\rightarrow\infty$ in (\ref{prb15}) we obtain
\[
m\geq\frac{1}{N}\int_{\mathbb{R}^{N}}\varphi\,d\mu\geq\frac{1}{N}\int
_{\{x_{i}\}}\varphi\,d\mu=\frac{1}{N}\mu(\{x_{i}\})\geq\frac{1}{N}S^{N/2}%
\]
which contradicts the Proposition \ref{lemb7}, showing that $J=\emptyset$.
Hence, it follows from Lions' lemma that
\[
u_{n}\rightarrow0\quad\mathrm{in}\,L_{loc}^{2^{\ast}}(\mathbb{R}^{N}).
\]
In particular, $u_{n}\rightarrow0$ in $L^{2^{\ast}}(\Omega_{\delta}),$ so that%
\[
0\leq\int_{\Omega_{\delta}}|u_{n}|^{p(x)}\mathrm{d}x\leq\int_{\Omega_{\delta}%
}\left\vert u_{n}\right\vert ^{p^{-}}\mathrm{d}x+\int_{\Omega_{\delta}%
}\left\vert u_{n}\right\vert ^{2^{\ast}}\mathrm{d}x\rightarrow0.
\]
Since $(u_{n})\subset\mathcal{N}$, we have
\begin{align*}
\lim_{n\rightarrow\infty}\left\Vert u_{n}\right\Vert _{1,2}^{2}  &
=\lim_{n\rightarrow\infty}\left(  \int_{\Omega_{\delta}}\left\vert
u_{n}\right\vert ^{p(x)}\mathrm{d}x+\int_{\mathbb{R}^{N}\setminus\Omega_\delta
}\left\vert u_{n}\right\vert ^{2^{\ast}}\mathrm{d}x\right) \\
&  =\lim_{n\rightarrow\infty}\int_{\mathbb{R}^{N}\setminus\Omega_{\delta}%
}\left\vert u_{n}\right\vert ^{2^{\ast}}\mathrm{d}x=:L.
\end{align*}
Thus, by making $n\rightarrow\infty$ in the equality
\[
I(u_{n})=\frac{1}{2}\left\Vert u_{n}\right\Vert _{1,2}^{2}-\int_{\Omega
_{\delta}}\frac{\left\vert u_{n}\right\vert ^{p(x)}}{p(x)}\mathrm{d}x-\frac
{1}{2^{\ast}}\int_{\mathbb{R}^{N}\setminus\Omega_{\delta}}\left\vert
u_{n}\right\vert ^{2^{\ast}}\mathrm{d}x
\]
we obtain
\[
0<m=\frac{1}{2}L-\frac{1}{2^{\ast}}L=\frac{1}{N}L.
\]

Since
\[
S\leq\frac{\left\Vert u_{n}\right\Vert _{1,2}^{2}}{\left\Vert u_{n}\right\Vert
_{2^{\ast}}^{2}}\leq\frac{\left\Vert u_{n}\right\Vert _{1,2}^{2}}{\left(
\int_{\mathbb{R}^{N}\setminus\Omega_{\delta}}\left\vert u_{n}\right\vert
^{2^{\ast}}\mathrm{d}x\right)  ^{2/2^{\ast}}}%
\]
we obtain $m=\frac{L}{N}\geq\frac{1}{N}S^{N/2}$, which contradicts Proposition
\ref{lemb7} and proves that $u\not \equiv 0$.

Now, combining the weak convergence
\[
u_{n}\rightharpoonup u\quad\mathrm{in}\,D^{1,2}(\mathbb{R}^{N})
\]
with the fact that $I^{\prime}(u_{n})\rightarrow0$ in $D^{-1}$ we conclude
that
\[
I^{\prime}(u)(v)=0\quad\forall\,v\in D^{1,2}(\mathbb{R}^{N}),
\]
meaning that $u$ is a nontrivial critical point of $I.$

Thus, taking into account Proposition \ref{lemb5}, in order to complete the
proof that $u$ is a ground state solution for (\ref{prb1}) we need to verify
that $I(u)=m$. Indeed, since
\[
I(u_{n})=\left(  \frac{1}{2}-\frac{1}{p^{-}}\right)  \Vert u_{n}\Vert
_{1,2}^{2}+\int_{\mathbb{R}^{N}}\left(  \frac{1}{p^{-}}-\frac{1}{p(x)}\right)
u_{n}^{p(x)}\mathrm{d}x,
\]
the weak convergence $u_{n}\rightharpoonup u$ in $D^{1,2}(\mathbb{R}^{N})$ and
Fatou's Lemma imply that%
\begin{align*}
m  &  \geq\left(  \frac{1}{2}-\frac{1}{p^{-}}\right)  \liminf_{n\rightarrow
\infty}\Vert u_{n}\Vert_{1,2}+\liminf_{n\rightarrow\infty}\int_{\mathbb{R}%
^{N}}\left(  \frac{1}{p^{-}}-\frac{1}{p(x)}\right)  u_{n}^{p(x)}\mathrm{d}x\\
&  \geq\left(  \frac{1}{2}-\frac{1}{p^{-}}\right)  \Vert u\Vert_{1,2}%
+\int_{\mathbb{R}^{N}}\left(  \frac{1}{p^{-}}-\frac{1}{p(x)}\right)
u^{p(x)}\mathrm{d}x\\
&  =I(u)-\frac{1}{p^{-}}I^{\prime}(u)(u)=I(u)\geq m,
\end{align*}
showing that $I(u)=m$.  

\section{The semilinear elliptic problem in a bounded domain\label{Sec4}}

In this section we consider the elliptic problem%
\begin{equation}
\left\{
\begin{array}
[c]{ll}%
-\Delta u=u^{p(x)-1},\quad u>0 & \mathrm{in}\quad G\\
u=0 & \mathrm{on}\quad\partial G,
\end{array}
\right.  \label{prb23}%
\end{equation}
where $G$ is a smooth bounded domain of $\mathbb{R}^{N}$, $N\geq3,$ and
$p:G\rightarrow\mathbb{R}$ is a continuous function verifying $\mathrm{(H}%
_{1}\mathrm{)}$ and an additional hypothesis $\mathrm{(H}_{2}\mathrm{),}$
which will be stated in the sequel.

We recall that the usual norm in $H_{0}^{1}(G)$ is given by
\[
\Vert u\Vert:=\Vert\nabla u\Vert_{2}=\left(  \int_{G}|\nabla u|^{2}%
\mathrm{d}x\right)  ^{\frac{1}{2}}.
\]
We denote the dual space of $H_{0}^{1}(G)$ by $H^{-1}$.

The energy functional $I:H_{0}^{1}(G)\rightarrow\mathbb{R}$ associated with
the problem (\ref{prb23}) is defined by
\[
I(u):=\frac{1}{2}\int_{G}\left\vert \nabla u\right\vert ^{2}\mathrm{d}%
x-\int_{G}\frac{1}{p(x)}(u^{+})^{p(x)}\,\mathrm{d}x.
\]
It belongs to $C^{1}(H_{0}^{1}(G),\mathbb{R})$ and its derivative is given by
\begin{equation}
I^{\prime}(u)(v)=\int_{G}\nabla u\cdot\nabla v\,\mathrm{d}x-\int_{G}%
(u^{+})^{p(x)-1}v\,\mathrm{d}x,\quad\forall\,u,v\in H_{0}^{1}(G).
\label{prb26}%
\end{equation}

Thus, a function $u\in H_{0}^{1}(G)$ is a weak solution of (\ref{prb23}) if,
and only if, $u$ is a critical point of $I$. Moreover, as in the previous
section, the nontrivial critical points of $I$ are positive.

In the sequel, we will use the same notations of Section \ref{Sec3}. Thus,%
\[
J(u):=I^{\prime}(u)(u)=\left\Vert u\right\Vert ^{2}-\int_{G}(u^{+}%
)^{p(x)}\,\mathrm{d}x,
\]
the Nehari manifold associated with (\ref{prb23}) is defined by%
\[
\mathcal{N}:=\{u\in H_{0}^{1}(G)\setminus\{0\}:J(u)=0\},
\]
and
\[
m:=\inf_{u\in\mathcal{N}}I(u).
\]
Arguing as in the proof of Proposition \ref{lemb2} we can guarantee that $m>0$.

\begin{definition}
\label{GS2}We say that $u\in\mathcal{N}$ is a \textit{ground state solution}
for (\ref{prb23}) if $I^{\prime}(u)=0$ and $I(u)=m.$
\end{definition}

We gather in the next lemma some results that can be proved as in Section
\ref{Sec3}.

\begin{lemma}
\label{lemb12} Assume $\mathrm{(H}_{1}\mathrm{)}$. We claim that:

\begin{enumerate}
\item[(i)] $J^{\prime}(u)(u)<0$ for all $u\in\mathcal{N}.$ (Thus, $J^{\prime
}(u)\neq0$, for all $u\in\mathcal{N}.$)

\item[(ii)] If $I(u_{0})=m,$ then $I^{\prime}(u_{0})=0.$ (Thus, $u_{0}$ is a
weak solution of (\ref{prb23}).)

\item[(iii)] There exists a sequence $(u_{n})\subset\mathcal{N}$ such that
$u_{n}\geq0$ in $G,$ $I(u_{n})\rightarrow m$ and $I^{\prime}(u_{n}%
)\rightarrow0$ em $H^{-1}$.
\end{enumerate}
\end{lemma}

Before establishing our main result in this section, we need to fix some
notation. Let $U\subset\mathbb{R}^{N}$ a bounded domain and define
\begin{equation}
S_{q}(U):=\inf\left\{  \frac{\Vert\nabla v\Vert_{L^{2}(U)}^{2}}{\Vert
v\Vert_{L^{q}(U)}^{2}}:v\in H_{0}^{1}(U)\setminus\{0\}\right\}  ,\quad1\leq
q\leq2^{\ast}. \label{SqBR}%
\end{equation}

It is well known that if $1\leq q<2^{\ast}$ then the infimum in (\ref{SqBR})
is attained by a positive function in $H_{0}^{1}(U).$ Actually, this follows
from the compactness of the embedding $H_{0}^{1}(U)\hookrightarrow L^{q}(U).$

Another well-known fact is that in the case $q=2^{\ast}$ the infimum in
(\ref{SqBR}) coincides with the best Sobolev constant, i.e.
\begin{equation}
S_{q}(U)=S:=\inf\left\{  \frac{\Vert\nabla v\Vert_{L^{2}(\mathbb{R}^{N})}^{2}%
}{\Vert v\Vert_{L^{2^{\ast}}(\mathbb{R}^{N})}^{2}}:v\in D^{1,2}(\mathbb{R}%
^{N})\setminus\{0\}\right\}  . \label{SqS}%
\end{equation}
Moreover, the infimum (\ref{SqS}) is not attained if $U$ is a proper subset of
$\mathbb{R}^{N}.$

Let us also define%

\[
g(q):=\left(  \frac{1}{2}-\frac{1}{q}\right)  S_{q}(U)^{\frac{q}{q-2} },\quad
q\in (2,2^{\ast}].
\]

Since the function $q\in[1,2^{\ast}] \rightarrow S_{q}(U)$ is continuous (see
\cite{EA12}), we have%

\begin{equation}
S_{q}(U)\rightarrow S_{2}(U)\quad\mathrm{as}\,q\rightarrow2\ (q>2)
\label{prb28}%
\end{equation}

If we consider $S_{2}(U)<1$, then $\log S_{2}(U)\leq0$ and since the function
$q\rightarrow g(q)$ is continuous, it follows that%

\[
g(q)=\left(  \frac{1}{2}-\frac{1}{q}\right)  \exp\left(  \frac{q}{q-2}\log
S_{q}(U)\right)  \rightarrow0,\quad\mathrm{as}\,q\rightarrow2,\ q>2.
\]

Taking into account that $S_{2^{\ast}}(U)=S$, we can easily check that
$g(2^{\ast})=\frac{1}{N}S^{\frac{N}{2}}.$ Thus, denoting%

\[
\bar{q}:=\min\{q\in(2,2^{\ast}]:g(q)=\frac{1}{N}S^{\frac{N}{2}}\}
\]

we have, by continuity,%

\[
g(q)<g(\bar{q}),\quad\forall\ q\in(2,\bar{q}).
\]

We have proved the following lemma.

\begin{lemma}
\label{leaux} If $S_{2}(U)<1$, then there exists $\bar{q}\in (2,2^{\ast
}]$ such that
\[
g(q)<g(\bar{q})=\frac{1}{N}S^{\frac{N}{2}},\quad\forall\ q\in(2,\bar{q}).
\]

\end{lemma}

The additional condition $\mathrm{(H}_{2}\mathrm{)}$ is the following, where
$\Omega$ and $p^{-}$ are defined in ($H_{1}$):

\begin{description}
\item[($H_{2}$)] There exists a subdomain $U$ of $\Omega$ such that
$S_{2}(U)<1$ and%
\[
p(x)\equiv q,\quad\forall\ x\in U,
\]

where $p^{-}\leq q<\min\{\bar{q},p^{+}\}$ and $\bar{q}$ is given by Lemma
\ref{leaux}.
\end{description}

\begin{lemma}
Assume that $\mathrm{(H}_{1}\mathrm{)}$ and $\mathrm{(H}_{2}\mathrm{)}$ hold.
Then $m<\frac{1}{N}S^{\frac{N}{2}}$.
\end{lemma}

\begin{proof}
Let $\phi_{q}\in H_{0}^{1}(U)$ denote a positive extremal function of
$S_{q}(U).$ Thus, $\phi_{p}>0$ in $U$ and
\[
S_{q}(U)=\frac{\Vert\nabla\phi_{q}\Vert_{L^{2}(U)}^{2}}{\Vert\phi_{q}%
\Vert_{L^{q}(U)}^{2}}.
\]
Let us define the function $\widetilde{\phi}_{q}\in H_{0}^{1}(G)$ by
\[
\widetilde{\phi}_{q}(x):=\left\{
\begin{array}
[c]{lll}%
\phi_{q}(x) & \mathrm{if} & x\in U\\
0 & \mathrm{if} & x\in G\setminus U.
\end{array}
\right.
\]
For each $t>0$ we have
\[
I(t\widetilde{\phi}_{q})=\frac{t^{2}}{2}\int_{G}\left\vert \nabla
\widetilde{\phi}_{q}\right\vert ^{2}\mathrm{d}x-\int_{G}\frac{t^{p(x)}}%
{p(x)}(\widetilde{\phi}_{q})^{p(x)}\,\mathrm{d}x=\frac{\alpha}{2}t^{2}%
-\frac{\beta}{q}t^{q},
\]
where%
\[
\alpha:=\int_{U}|\nabla\phi_{q}|^{2}\mathrm{d}x\quad\mathrm{and}\quad
\beta:=\int_{U}(\phi_{q})^{q}\,\mathrm{d}x.
\]
Fixing
\[
t_{q}:=(\alpha/\beta)^{\frac{1}{q-2}}%
\]
it is easy to see that $t_{q}\widetilde{\phi}\in\mathcal{N}$ and
\begin{equation}
I(t_{q}\widetilde{\phi}_{q})=\left(  \frac{1}{2}-\frac{1}{q}\right)  \left(
\frac{\alpha^{q}}{\beta^{2}}\right)  ^{\frac{1}{q-2}}=\left(  \frac{1}%
{2}-\frac{1}{q}\right)  S_{q}(U)^{\frac{q}{q-2}}. \label{prb27}%
\end{equation}

Since $S_{2}(U)<1$ and $p^{-}\leq q<\min\{p^{+},\bar{q}\}$, it follows from
Lemma \ref{leaux} that
\[
g(q)=I(t_{q}\widetilde{\phi}_{q})<\frac{1}{N}S^{\frac{N}{2}}.
\]
This implies that $m<\frac{1}{N}S^{\frac{N}{2}}.$
\end{proof}

In the next, we present sufficient conditions for the inequality $S(U)<1$ to
hold when $U$ is either a ball or an annulus. We will denote by $B_{R}(y)$ the
ball centered at $y$ with radius $R>0.$ When $y=0$ we will write simply
$B_{R}.$

\begin{example}
Let $U=B_{R}(y)\subset\Omega.$ Since the Laplacian operator is invariant under
translations, $S_{2}(B_{R})=S_{2}(B_{R}(y))$. Moreover, a simple scaling
argument yields
\begin{equation}
S_{2}(B_{R})=R^{-2}S_{2}(B_{1}). \label{prb29}%
\end{equation}
So, if $R>S_{2}(B_{1})^{\frac{1}{2}}$ then $S_{2}(U)=S_{2}(B_{R}(y))<1.$
\end{example}

\begin{example}
Let $U=A_{R,r}:=B_{R}(y)\setminus\overline{B_{r}(z)}\subset\Omega$, with
$\overline{B_{r}(z)}\subset B_{R}(y)$, for some $y,z\in\Omega$ and $R>r>0$.
Since the Laplacian operator is invariant under orthogonal transformations, it
can be seen that
\[
S_{2}(B_{R}\setminus\overline{B_{r}(se_{1})})=S_{2}(A_{R,r})
\]
for some $s\in\lbrack0,R-r)$ where $e_{1}$ denotes the first coordinate
vector. According to Proposition 3.2 in \cite{K03}, the function $t\rightarrow
S_{2}(B_{R}\setminus\overline{B_{r}(te_{1})})$ is strictly decreasing for
$t\in\lbrack0,R-r).$ Therefore,
\begin{equation}
S_{2}(B_{R}\setminus\overline{B_{r}})>S_{2}(A_{R,r}). \label{prb32}%
\end{equation}
Since $B_{(R-r)/2}$ is the largest ball contained in $B_{R}\setminus
\overline{B_{r}}$ we have
\begin{equation}
S_{2}(B_{R}\setminus\overline{B_{r}})<S_{2}(B_{(R-r)/2})=\left(  \frac{R-r}%
{2}\right)  ^{-2}S_{2}(B_{1}). \label{prb33}%
\end{equation}
Hence, if $R-r>2S_{2}(B_{1})^{\frac{1}{2}}$ then (\ref{prb32}) and
(\ref{prb33}) imply that $S_{2}(U)=S_{2}(A_{R,r})<1.$
\end{example}

Thus, we can replace the condition $S(U)<1$ in \textsf{(}$\mathrm{H}_{2}%
$\textsf{)} by either $R>S_{2}(B_{1})^{\frac{1}{2}}$ when $U=B_{R}(y)$ or
$R-r>2S_{2}(B_{1})^{\frac{1}{2}}$ when $U=A_{R,r}.$

The main result this section is the following.

\begin{theorem}
\label{lemb17} Assume \textrm{(}$\mathrm{H}_{1}$\textrm{)} and \textrm{(}%
$\mathrm{H}_{2}$\textrm{)}. Then the problem (\ref{prb23}) has at least one
positive ground state solution.
\end{theorem}

\begin{proof}
According to item \textrm{(iii)} of Lemma \ref{lemb12}, there exists a
sequence $(u_{n})\subset\mathcal{N}$ satisfying $I(u_{n})\rightarrow m$ and
$I^{\prime}(u_{n})\rightarrow0$ em $H^{-1}$. Since $(u_{n})$ is bounded in
$H_{0}^{1}(G)$, there exist $u\in H_{0}^{1}(G)$ and a subsequence, still
denoted by $\left(  u_{n}\right)  ,$ such that $u_{n}\rightharpoonup u$ in
$H_{0}^{1}(G),$ $u_{n}\rightarrow u$ in $L^{p}(G),$ for $1\leq p<2^{\ast},$
and $u_{n}(x)\rightarrow u(x)\,\mathrm{a.e.}$ in $G.$ Arguing as in Section
\ref{Sec3}, we can prove that $u\not \equiv 0$, $I^{\prime}(u)=0$ and
$I(u)=m$, showing thus that $u$ is a ground state solution of (\ref{prb23}).
\end{proof}

\section{Acknowledgements}

C. O. Alves was partially supported by CNPq/Brazil (304036/2013-7) and
INCT-MAT. G. Ercole was partially supported by CNPq/Brazil (483970/2013-1 and
306590/2014-0) and Fapemig/Brazil (APQ-03372-16).

%\subsection{Subsection}\label{sec:nada}

%\subsubsection{Subsubsection}\label{sec:nada2}

%Bibliograf�a.
%-----------------------------------------------------------------

\end{document}